\documentclass[final,a4paper,leqno]{siamltex}
\usepackage[english]{babel}
\usepackage{amsmath,amssymb,amsfonts,graphicx,subcaption}
\usepackage{myshortcuts,float}
\usepackage{fancyhdr}
\usepackage{MnSymbol}
\usepackage{csvsimple}
\usepackage{tikz}
\usepackage{pgfplots}
\usetikzlibrary{decorations.pathreplacing}

\DeclareMathOperator*{\argmin}{argmin}
\newcommand\vectorize{\mathord{\mathrm{vec}}}
\usepackage{appendix}
\usepackage{comment}
\newcommand{\norm}[1]{\left\lVert#1\right\rVert}
\usepackage[linesnumbered,lined,boxed,commentsnumbered,ruled,vlined]{algorithm2e}
\let\oldnl\nl% Store \nl in \oldnl
\newcommand{\nonl}{\renewcommand{\nl}{\let\nl\oldnl}}
\newcommand\numberthis{\addtocounter{equation}{1}\tag{\theequation}}

\usepackage[hidelinks]{hyperref} 
\usepackage[labelfont=bf]{caption}
\iffalse % set to iftrue if you want to regen PDF-files in gfx_tmp.
  \usepackage{tikz}
  \usepackage{pgfplots}
  \pgfplotsset{
    compat=newest,
    tick label style={font=\scriptsize},
    label style={font=\scriptsize},
    legend style={font=\scriptsize}
  }
   \usetikzlibrary{decorations.pathreplacing}
  \iffalse % Set to iftrue if you want to regen fig PDF-files.
     \usepgfplotslibrary{external} % requires a recent version of pgf package
  \else
     \renewcommand{\tikzsetnextfilename}[1]{}
  \fi
\else
  \usepackage{tikzexternal}
  \tikzsetexternalprefix{gfx_tmp/}
  
\fi
\title{Infinite GMRES for parameterized linear systems}

\author{Elias Jarlebring,
Siobhán Correnty\thanks{Department of Mathematics, Royal Institute of Technology (KTH), Stockholm, SeRC Swedish
e-Science Research Center, email: 
\{\texttt{eliasj,correnty}\}\texttt{@kth.se}}}

\begin{document}

\maketitle
\begin{abstract}
We consider linear parameter-dependent systems $A(\mu) x(\mu) = b$ for many different $\mu$, where $A$ is large and sparse, and depends nonlinearly on $\mu$. Solving such systems individually for each $\mu$ would require great computational effort. In this work we propose to compute a partial parameterization $\tilde{x} \approx x(\mu)$ where $\tilde{x}(\mu)$ is cheap to compute for many different $\mu$. Our methods are based on the observation that a companion linearization can be formed where the dependence on $\mu$ is only linear. In particular, we develop methods which combine the well-established Krylov subspace method for linear systems, GMRES, with algorithms for nonlinear eigenvalue problems (NEPs) to generate a basis for the Krylov subspace. Within this new approach, the basis matrix is constructed in three different ways, using a tensor structure and exploiting that certain problems have low-rank properties. We show convergence factor bounds obtained similarly to those for the method GMRES for linear systems. More specifically, a bound is obtained based on the magnitude of the parameter $\mu$ and the spectrum of the linear companion matrix, which corresponds to the reciprocal solutions to the corresponding NEP. Numerical experiments illustrate the competitiveness of our methods for large-scale problems. The simulations are reproducible and publicly available online.
\end{abstract}

\begin{keywords}
parameter-dependent linear systems, Krylov methods, companion linearization, shifted linear systems, infinite Arnoldi, low-rank
\end{keywords}

\section{Introduction} \label{sec:introduction}
We are interested in the numerical solution of the large complex linear system
\begin{align} \label{eq:our-problem}
	A (\mu) x (\mu) = b,
\end{align}
with $A(\mu) \in \CC^{n \times n}$ analytic, $A(0)$ nonsingular and $b \in \CC^n$, for many possible values of $\mu \in \CC \backslash \{ 0 \}$.

Solving large linear systems of equations in an efficient manner is a topic of great importance in most scientific fields involving comxputation. Linear systems often arise from the application of the finite element method (FEM) to a partial differential equation (PDE), leading to a large and sparse matrix $A$. Since a finer FEM-discretization leads to higher accuracy and better modelling possibilities but also to larger systems, improving efficiency, accuracy and robustness of methods is critical \cite{Estep1996ComputationalDE}.

We consider problems where there is a nonlinear dependence on a parameter, yet we are interested in obtaining solutions for many different $\mu$ simultaneously. Such problems arise naturally in many situations; for example in
the study of PDEs with uncertainty, a parameter-dependent system which we
will consider
is the Helmholtz equation with a parameterized material coefficient.
Since the parameter $\mu$ is not known in advance, the problem can be seen as an uncertainty quantification problem in the sense of \cite{UncertCompTensRep}.
Problems of this type can also appear in the context of model reduction, where $\mu$ is usually the Laplace variable. We present an example in this direction
in Section~\ref{sec:TDS}. See \cite{Benner:2021:MOR} and  \cite{Benner2015UncertaintyQF} for samples of applications and literature on model order reduction.
In contrast to many model reduction techniques, our setting does not lead to
a parameterization which is of the same structure as the original problem, but we aim to obtain any computationally cheap parameterization of the solution vector. This freedom allows
us to completely generalize iterative methods for linear systems.

The new methods we present are based on the well-established GMRES method  for linear systems \cite{Saad:1986:GMRES}. The main idea of our approach can be summarized as follows.
For our derivation we define a new function
\begin{align} \label{eq:Bdef}
  B(\mu):=\frac{1}{\mu}A(0)^{-1}\big( A(0)-A(\mu) \big)
\end{align}
such that \eqref{eq:our-problem} can equivalently be written as
\begin{align}
  \big(\mu B(\mu)-I \big) x(\mu)=-A(0)^{-1}b.
  \label{eq:we-solve}
\end{align}
Let $B_N$ be the truncated Taylor series expansion of $B$:
\begin{align}
   B_N(\mu) = \frac{1}{0!} B_0 + \frac{1}{1!} B_1 \mu + \frac{1}{2!} B_2 \mu^2 + \dots + \frac{1}{N!} B_N \mu^N,
   \label{eq:B-of-mu}
\end{align}
where $B_i = -\frac{1}{i+1}A(0)^{-1} A^{(i+1)}(0) \in \CC^{n \times n}$. In order to handle the nonlinearity in \eqref{eq:we-solve}, we use a technique called companion linearization which is very commonly used to analyze polynomial eigenvalue problems (PEPs).
As we show in  Theorem~\ref{solution}, equation \eqref{eq:we-solve} is equivalent to the following linear system where $\mu$ only appears linearly:
\begin{align} \label{eq:Blinsys}
(\mu \mathbf{B}_N - I) v(\mu) = c.
\end{align}
The constant matrix $\mathbf{B}_N$ and constant vector $c$ are defined as
\begin{align}
    \mathbf{B}_N :=
    \begin{bmatrix}
    B_0&  \cdots&\cdots& \cdots& B_N \\
    I&&&&0\\
    &\frac{1}{2}I&&&\vdots\\
     &&\ddots&&\vdots\\
          &&&\frac{1}{N}I&0\\
    \end{bmatrix} \in \CC^{(N+1)n \times (N+1)n}
    \label{comp_matrix}
\end{align}
and
\begin{align}
  c:= -e_1\otimes A(0)^{-1} b  \in \CC^{(N+1)n},
  \label{eq:RHS}
\end{align}
where $\otimes$ denotes the Kronecker product, and $e_1$ is the
first unit vector.
%\begin{align}
%	c := [-A(0)^{-1} b, 0, \ldots, 0]^T \in \CC^{(N+1)n}.
%	\label{eq:RHS}
%\end{align}
In Section~\ref{sec:linearization} we completely describe this equivalence,
analogous to the theory of PEPs, where $\mathbf{B}_N$ is
called a companion matrix.

If we apply GMRES on the linear system \eqref{eq:we-solve}, we need to build the Krylov subspace
\begin{align} \label{eq:our-krylov}
	\mathcal{K}_k (\mathbf{B}_N, c) := \text{span} \{c, \mathbf{B}_N c, \ldots, 	\mathbf{B}_N^{k-1} c \}.
\end{align}
The same matrix $\mathbf{B}_N$  appears frequently in the field of nonlinear eigenvalue problems (NEPs),
where the same Krylov subspace \eqref{eq:our-krylov}
is used to construct numerical methods.
%has a particular structure.
The Krylov subspace has a particular structure which
has been heavily exploited for NEP methods, e.g. the infinite Arnoldi method \cite{Jarlebring:2012:INFARNOLDI} and various improvements thereof \cite{doi:10.1002/nla.2043,Jarlebring:2014:SCHUR,DelayProblem,TensorArnoldi,Mele2018}. We will now show that many of the same techniques can be applied here.
In particular,
%we represent $B(\mu)$ in a finite number of
we can let $N\rightarrow \infty$ but still carry out the algorithm
with a finite number of
linear algebra operations without truncation error in $B_N(\mu)$.

To our knowledge, this work contains the first result that exploits the connection between linear systems and NEPs in a way that allows the generalization
of iterative methods for linear systems to parameterized linear systems.
Although we have focused on the flavors of the NEP-method
infinite Arnoldi method, there are many algorithms for NEPs that may also lead to competitive approaches, e.g., CORK \cite{VanBeeumen:2015:CORK} or TOAR \cite{Kressner:2014:TOAR}.

Krylov subspaces are invariant under shifts, a property previously shown to be useful in other works on parameterized linear systems \cite{Baumann2015NestedKM, GuSimoncini, KressnerToblerLR, etna_vol45_pp499-523,SOODHALTER2014105,Bahkos:2017:MULTIPREC}. Our linearization allows us to consider just one Krylov subspace formed independently of $\mu$ by a linear combination of the power sequence associated with the matrix $\mathbf{B}_N$ and the vector $c$ as in \eqref{eq:our-krylov}. We reuse the associated basis matrix to compute $\tilde{x}(\mu)$ for all values of $\mu$, solving a least squares problem for every value of $\mu$.

The first new method we propose in this paper considers an efficient way of handling the Krylov basis matrix as in \cite{TensorArnoldi} while the second is specifically designed to handle $A(\mu)$ with higher order terms of reduced rank, ultimately allowing for a more efficient construction of the Krylov basis matrix analogous to the low-rank \cite{doi:10.1002/nla.2043}. We provide convergence theory and show that convergence factor bounds can be obtained in a way similar to the linear case. Eigenvalue based bounds are derived from the solution of NEPs. %Our methods show competitive computational results.

The paper is organized as follows. In the next section, we explain the linearization we used and prove that we can easily recover the solution we seek from this linearization. In Section~\ref{sec:algorithms} we present our new methods, Infinite GMRES and Low-Rank Infinite GMRES. In Section~\ref{sec:convergence} we show our convergence theory and in Section~\ref{sec:simulations} we provide numerical examples which illustrate the theory.

\section{Companion Linearization} \label{sec:linearization}
Companion linearizations has been extensively used for PEPs
(see, e.g.,  \cite{Mackey:2006:VECT}) but also
for linear systems \cite{GuSimoncini}.
We now show how companion linearization can be applied leading to
a matrix structure that is also present in the infinite Arnoldi method,
where it was be used to dynamically expand the linearization, further explained
 in Section~\ref{sec:algorithms}.
Let $B_N(\mu)$ be the truncated Taylor expansion of $B(\mu)$ as in \eqref{eq:B-of-mu} and denote
%We explain in Section~\ref{sec:algorithms} that we actually represent this expansion without truncation error. We define
\begin{align} \label{eq:D-mat}
D_{N,N+1} :=
	\begin{bmatrix}
	1 & & & & 0 \\
	& \frac{1}{2} & & & \vdots \\
	& & \ddots & & \vdots \\
	& & & \frac{1}{N} & 0
	\end{bmatrix}
	\in \RR^{N \times (N+1)}.
\end{align}
The companion linearization can be explicitly expressed as follows.
\begin{theorem} \label{solution}
  Suppose a given parameter-dependent matrix $A(\mu)$ is as in \eqref{eq:our-problem}.
  %Define a function $B(\mu)$ by \eqref{eq:Bdef} and
  Let $B_N(\mu)$ be a truncated expansion of $B$ as in \eqref{eq:B-of-mu}.
  %and a fixed right-hand side vector $b \in \CC^n$.
  For any $\mu \in \CC \backslash \{ 0 \}$, consider the
  linear system \eqref{eq:Blinsys}, where $\mathbf{B}_N$ as in \eqref{comp_matrix}, $c$ as in \eqref{eq:RHS} and
\begin{align*}
	 v(\mu) := [v_0(\mu),\ldots,v_{N}(\mu)]^T \in \CC^{(N+1)n}.
\end{align*}
Let $A_N(\mu)$ refer to the Taylor series expansion of $A(\mu)$ truncated after $(N+1)$ terms.

Then, the linear system \eqref{eq:Blinsys} and
    \begin{align}
    A_N(\mu)x_N(\mu)=b
    \label{eq:ANrelation}
\end{align}
are equivalent in the following sense:
\begin{itemize}
    \item[(a)] Solutions to the linear system \eqref{eq:Blinsys} are of the form $v_i(\mu)=\frac{\mu^i}{i!} x_N(\mu)$, $i=0,\ldots,N$, where $x_N(\mu) \in \CC^n$ satisfies \eqref{eq:ANrelation}
    \item[(b)] Let $x_N(\mu) \in \CC^n$ be a solution to \eqref{eq:ANrelation}, then $v_i(\mu):=\frac{\mu^i}{i!} x_N(\mu)$ is a solution to \eqref{eq:Blinsys}
\end{itemize}
\end{theorem}
\begin{proof}
a) We assume \eqref{eq:Blinsys} holds where $v_i(\mu)=\frac{\mu^i}{i!} x_N(\mu)$ and we look at the first line in this matrix vector product:
\begin{align*}
    -A(0)^{-1}b &= \mu B_0 \frac{\mu^0}{0!} x_N(\mu) + \mu B_1 \frac{\mu^1}{1!} x_N(\mu) + \cdots + \mu B_N \frac{\mu^N}{N!} x_N(\mu) - \frac{\mu^0}{0!} x_N(\mu) \\
    &= \big( \mu B_N(\mu) - I \big) x_N(\mu) \\
    &= \left( \mu \Big( \frac{1}{\mu} A(0)^{-1} \big( A(0) - A_N(\mu) \big) \Big) - I \right) x_N(\mu).
\end{align*}
Therefore, \eqref{eq:ANrelation} holds.

b) From \eqref{eq:ANrelation} we directly conclude that
\begin{align}
        \left(
        \mu
        \begin{bmatrix}
        B_0 & B_1 & \cdots & \cdots & B_N \\
        I & & & & 0 \\
        & \frac{1}{2} I & & & \vdots \\
        & & \ddots & & \vdots \\
        & & & \frac{1}{N} I & 0 \\
        \end{bmatrix}
        -
        \begin{bmatrix}
        I & & & & \\
        & \ddots & & & \\
        & & \ddots & & \\
        & & & \ddots & \\
        & & & & I \\
        \end{bmatrix}
        \right)
        \begin{bmatrix}
        \frac{\mu^0}{0!} x_N(\mu) \\
        \vdots \\
        \vdots \\
        \vdots \\
        \frac{\mu^N}{N!} x_N(\mu) \\
        \end{bmatrix}
        =
        \begin{bmatrix}
        -A(0)^{-1}b \\
        0 \\
        \vdots \\
        \vdots \\
        0
        \end{bmatrix}.
        \label{eq:b_equation}
\end{align}
Block rows $2,3,\ldots,N+1$ above can be rewritten with Kronecker products and then simplified as follows. By defining
\begin{align*}
    S &:=
    \begin{bmatrix}
    0 & 1 & & \\
	& \ddots & \ddots & \\
	& & 0 & 1 \\
    \end{bmatrix} \in \RR^{N \times (N+1)}, \\
    \Bar{\mu} &:= [\frac{ \mu^0}{0!}, \frac{ \mu^1}{1!}, \ldots, \frac{ \mu^N}{N!}]^T \in \RR^{N+1}
\end{align*}
we have
\begin{align*}
    \left( (\mu D_{N,N+1} - S) \otimes I \right)
    (\Bar{\mu}
    \otimes x_N(\mu)) &=
    (\mu D_{N,N+1} - S) \Bar{\mu} \otimes x_N(\mu) = 0,
\end{align*}
since $(\mu D_{N,N+1} - S) \Bar{\mu} = 0$, which corresponds to rows $2,\ldots,N+1$ in \eqref{eq:Blinsys}.

To show that the first $n$ equations in \eqref{eq:Blinsys} are satisfied, we now consider the first block row of \eqref{eq:b_equation}. The difference between the left-hand side and the right-hand side gives us
\begin{align*}
    \mu \Big(\frac{ \mu^0}{0!} B_0 + \ldots + \frac{\mu^N}{N!} B_N \Big) x(\mu) - x(\mu) + A(0)^{-1}b &= \Big( \mu B_N (\mu) - I \Big) x_N(\mu) + A(0)^{-1}b = 0
\end{align*}
due to \eqref{eq:we-solve}.
\end{proof}

\section{Algorithms} \label{sec:algorithms}
\subsection{GMRES for the shifted system} \label{sec:GMRES-background}
As a preparation for the algorithm derivation, we consider a shifted parameter-dependent linear equation system given generally as
\begin{align}
    (\mu C - I) x(\mu) = b,
    \label{eq:transformed_problem}
\end{align}
where $C \in \CC^{p \times p}$, $\mu \in \CC \backslash \{ 0 \}$,  $x(\mu) \in \CC^{p}$ and $b \in \CC^{p}$. GMRES for this type of shifted systems is derived from the standard GMRES method for solving $Cx=b$  summarized below;
see, e.g. \cite{Saad:1986:GMRES}.

On the $m$-th iteration of GMRES, we form an Arnoldi factorization, consisting of matrices $Q_m \in \CC^{p \times m}$, $Q_{m+1} \in \CC^{p \times (m+1)}$ and $\underline{H}_m \in \CC^{(m+1) \times m}$ that satisfy
\begin{align*}
    C Q_m = Q_{m+1} \underline{H}_m,
    %\label{arnoldi_fac}
\end{align*}
where $\underline{H}_m$ is an upper Hessenberg matrix and $Q_m = (q_1,\ldots,q_m) \in \CC^{p \times m}$ is a matrix whose columns form an orthonormal basis for the Krylov subspace of dimension $m$ associated with matrix $C$ and vector $b$, defined as
\begin{align*}
    \mathcal{K}_m (C,b) := \text{span} \{b, C b, \dots ,C^{m-1} b \}.
\end{align*}
In practice, we perform one matrix vector product, $y = C q_m$ with $q_1 = b/ \norm{b}$, and orthogonalize this vector against the columns of $Q_m$ by a Gram-Schmidt process. This new vector is used to form $Q_{m+1}$. The orthogonalization coefficients are stored in $\underline{H}_m$.

Due to the shift-invariance property of Krylov subspaces, we have
\begin{align*}
    \mathcal{K}_m (\mu C-I,b) = \mathcal{K}_m (C,b).
    %\label{eq:equiv-krylov}
\end{align*}
Thus, we can directly form an Arnoldi factorization for $\mu C - I$:
\begin{align*}
	\big(\mu C - I \big) Q_m = Q_{m+1}  \big( \mu \underline{H}_m - \underline{I}_m \big).
\end{align*}
This is essentially the relation pointed out in \cite[Equation~(2.4)]{Baumann2015NestedKM}.
Therefore, the $m$-th iterate of GMRES for \eqref{eq:transformed_problem} is found by solving the shifted least squares problem
\begin{align}
    x_m(\mu) = Q_m \big( \argmin_{x \in \CC^m} \norm{ \big( \mu \underline{H}_m - \underline{I}_m \big) x - e_1 \norm{b} } \big)
    \label{eq:least-squares}
\end{align}
where $\underline{I}_m$ is the identity matrix of size $m \times m$ with an extra row of zeros added at the bottom. This least squares problem is equivalent to finding the vector $x_m(\mu) \in \mathcal{K}_m(C,b)$ which minimizes the residual of \eqref{eq:transformed_problem}. Thus, solving \eqref{eq:transformed_problem} for a set $\boldsymbol{\mu} = \{ \mu_i \}_{i=1}^j $ reduces to constructing one Arnoldi factorization for the entire set, followed by a least squares problem for each $\mu_i$.

The overdetetermined \eqref{eq:least-squares} is computationally cheap since it is small in comparison to the size of the original problem. In practice \eqref{eq:least-squares} can be computed using Givens rotators as the matrix is a Hessenberg matrix, as pointed out e.g. in \cite{Saad:1986:GMRES}. We leave out the Givens rotator in our algorithms below since other parts of the algorithm are computationally dominating, although it could also be used here.

\subsection{An extension to infinity and infinite vs. finite}  \label{sec:build-krylov}

The Arnoldi method is efficient only if we can compute the corresponding matrix vector products efficiently. We consider a special representation of this product using the linearization in \eqref{eq:Blinsys}. This is completely
analogous as the infinite Arnoldi method \cite{Jarlebring:2012:INFARNOLDI}, and the proofs in this section are omitted for brevity.
Denote the $i$-th column of matrix $X$ by $x_i \in \CC^n$ and the operation of stacking the columns of a matrix into a vector by $\text{vec}(X) := (x_1^T, x_2^T, \dots , x_{N+1}^T)^T$.

\begin{lemma}
Let $\mathbf{B}_N$ be as in \eqref{comp_matrix}. For any $X \in \CC^{n \times (N+1)}$,
\begin{align*}
    \mathbf{B}_N \vectorize(X) = \vectorize(\tilde{x}, X D_{N+1,N})
\end{align*}
where
\begin{align*}
    \tilde{x} = -A(0)^{-1} \left(\sum_{i=1}^{N+1} \frac{1}{i} A^{(i)}(0) x_i \right).
\end{align*}
\end{lemma}

%We note that since $A(0)^{-1}$ does not change, we can compute the LU factorization before beginning the algorithm. In this way, the linear system can be solved efficiently in every iteration.

As we mentioned in section~\ref{sec:GMRES-background}, the method GMRES forms the Arnoldi factorization with the right-hand side vector. In the case of solving \eqref{eq:Blinsys}, the right-hand side vector \eqref{eq:RHS} has only $n$ non-zero entries, located in the first block. This corresponding matrix vector product can be represented in a specific way. We consider the following theorem.
\begin{theorem} \label{matvec}
Let $\mathbf{B}_N$ be as in \eqref{comp_matrix}. Suppose
\begin{align*}
    X = (\hat{X},0, \dots ,0) \in \CC^{n \times (N+1)}
\end{align*}
and $\hat{X} \in \CC^{n \times k}$, $k < N$. Then,
\begin{align*}
    \mathbf{B}_N \vectorize{(X)} = \vectorize{(\tilde{x}, \hat{X} D_{k,k},0, \dots ,0)},
\end{align*}
where
\begin{align} \label{eq:y-tilde}
    \tilde{x} = -A(0)^{-1} \big( \sum_{i=1}^k \frac{1}{i} A^{(i)}(0) x_i \big).
\end{align}
\end{theorem}
A direct result of Theorem~\ref{matvec} is the structure of the resulting vector is changed only by an expanding the number of non-zero entries in the first block and the number of floating point operations to compute $\mathbf{B}_N \vectorize{(\hat{X},0,...,0)}$ is independent of the number of zero elements. Hence, we can take the product of a vector with an infinite tail of zeros with an infinite companion matrix in a finite number of linear algebra operations, thus representing the Taylor series \eqref{eq:B-of-mu} without any truncation error.

Let $X_1,\dots,X_m$ be the matrix version of the vectors after $m$ iterations of the Arnoldi method. We note that the tailing zeros of the new vector $\mathbf{B}_N \vectorize{(X_m)}$ are preserved after orthogonalization against $X_1, \dots ,X_m$. As the Arnoldi method consists of just these operations, the method described above is suitable.

\subsection{Generalizations of Infinite Arnoldi method}
Based on the results presented in Section~\ref{sec:build-krylov}
we can directly state a generalization of GMRES, summarized in  Algorithm~\ref{alg:inf_gmres}.
%we show a method for solving the parameter-dependent system \eqref{eq:our-problem}.
This algorithm consists of an Arnoldi factorization, where the matrix vector products are performed as in Section~\ref{sec:build-krylov}. The resulting vectors are orthogonalized by a Gram-Schmidt process. The approximate solution is the result of a least squares problem as described in \eqref{eq:least-squares}. The error from this method thus comes entirely from GMRES, i.e., we extend the expansion to infinity without approximation error. We refer to this method as infinite GMRES. The complete algorithm is included in Algorithm~\ref{alg:inf_gmres}.

\subsubsection{Low-rank Infinite GMRES}\label{sect:LR_infgmres}
The infinite Arnoldi method as presented in \cite{Jarlebring:2012:INFARNOLDI}
has been extended and improved in various ways. We illustrate
how the low-rank exploitation \cite{doi:10.1002/nla.2043} can be adapted to this setting.
Suppose now that $A(\mu) \in \CC^{n \times n}$ analytic and the higher order terms are of reduced rank, i.e.,
\begin{align}
    A(\mu) &= \sum_{i=0}^s \frac{1}{i!} A_i \mu^i + \sum_{i=s+1}^{\infty} \frac{1}{i!} U_i V^T \mu^i,
    \label{eq:A_reduced_rank}
\end{align}
where $A_i = A^{(i)}(0) \in \CC^{n \times n}$, $U_i = U F^{(i)}(0) \in \CC^{n \times p}$, $U,V\in\CC^{n\times p}$, and $F^{(i)}(0) \in \CC^{p \times p}$.

We can approximate the solution to \eqref{eq:our-problem} using our linearization as in \eqref{eq:Blinsys}, exploiting this structure in particular.
\begin{corollary}
Suppose the higher order terms of $B(\mu)$ as in \eqref{eq:Bdef} are of reduced rank and denote
\begin{align*}
    B_i
    = -\frac{1}{i+1} A_0^{-1} \big(U_{i+1} V^T \big)
    = \tilde{U}_{i} V^T, i = s,s+1,\ldots.
\end{align*}
If \eqref{eq:Blinsys} holds, then
\begin{align}\label{eq:lowrank_sys}
    \big( \mu
    \tilde{\mathbf{B}}_N -
    I \big)
    \tilde{v}(\mu)
    =
   c,
\end{align}
where
\begin{align} \label{eq:Bn-tilde} %\numberthis
\tilde{\mathbf{B}}_N &:=    \begin{bmatrix}
    B_0 & \cdots &\cdots&B_{s-1} & \tilde{U}_{s} & \cdots & \cdots &\tilde{U}_{N} \\
    \frac{1}{1}I & & & & \\
    &\ddots & & & & \\
    &&\frac{1}{s-1}I & & & & \\
     &&& \frac{1}{s} V^T & & & & \\
     &&& & \frac{1}{s+1} I_p & & & \\
     & &&& & \ddots & & \\
     & &&& & & \frac{1}{N} I_p & 0
    \end{bmatrix}  %\in \mathbb{R}^{\ell \times m},
\end{align}
and
\begin{align*}
\tilde{v}(\mu) &:=
    \left[
    \frac{\mu^{0}}{0!} x(\mu),
    \cdots,
    \frac{\mu^{s-1}}{(s-1)!} x(\mu),
    \frac{\mu^{s}}{s!} V^T x(\mu),
    \cdots,
    \frac{\mu^N}{N!} V^T x(\mu)
    \right]^T. %\in \mathbb{R}^{m},
\end{align*}
%with $\ell=\big( n+(s-1)n + p + (N-s)p \big)$, $m=\left( sn + \big(N-(s-1)\big)p \right)$ and $c$ as in \eqref{eq:RHS}.
\label{low_rank_cor}
\end{corollary}
\begin{proof}
The equation \eqref{eq:lowrank_sys} is obtained directly by multiplying
\eqref{eq:Blinsys} from the left with the block diagonal matrix
\[
\begin{bmatrix}
I\\
&\ddots&\\
&&I\\
&&&V^T\\
&&&&\ddots& \\
&&&&&V^T
\end{bmatrix}.
\]
To show the converse is established by noting that the first block row in \eqref{eq:lowrank_sys} multiplied by $A(0)^{-1}$, is the equation \eqref{eq:ANrelation}.
\end{proof}

In every loop of Algorithm~\ref{alg:low_rank_inf_gmres}, we must compute a new matrix vector product to be used in the expansion of $Q_m$ and $\underline{H}_m$. In theory, we must multiply the nonzero block of the vector $q_m$ by the relevant parts of $\tilde{\mathbf{B}}_N$ in \eqref{eq:Bn-tilde} and orthogonalize the resulting nonzero block against $(q_1,\ldots,q_m)$. In practice, since we do not store an infinite tail of zeros, we compute the new vector $q_{m+1}$ as follows.

Let
\begin{align}
	y = \vectorize{(\tilde{x}, X_1 D_{s-1,s-1}}, \frac{1}{s}V^T X_2, X_3 \tilde{D}_{s+1,m}),
	\label{eq:y-lr}
\end{align}
where
\begin{equation}
	\tilde{x} = -(A_0)^{-1} \left( \sum_{i=1}^s \frac{1}{i} A_i x_i + \sum_{i=s+1}^m \frac{1}{i} U_i x_i \right),
	\numberthis
	\label{eq:x-tilde-lr}
\end{equation}
$A_i$ and $U_i$ according to \eqref{eq:A_reduced_rank},
\begin{subequations}
\begin{eqnarray}\label{eq.xs1-lr}
	\vectorize{(X_1)} &=& q_m \big(1:(s-1)n \big), \\
	X_2 &=& q_m \big( (s-1)n+1:sn \big), \\
	\vectorize{(X_3)} &=& q_m \big( sn+1:sn + p(m-s) \big),
\end{eqnarray}
\end{subequations}
with
\begin{align}
x_i = \begin{cases}
X_{1}(:,i) \in \CC^{n}& i=1,\ldots,s-1 \\
X_2 \in \CC^{n} & i = s \\
X_{3}(:(i-s)) \in \CC^{p} & i=s+1,s+2,\ldots,
\end{cases}
\numberthis
\label{eq:xs2-lr}
\end{align}
\begin{align}
\tilde{D}_{i,j} :=
	\begin{bmatrix}
	\frac{1}{i} &&& \\
	& \ddots && \\
	&& \ddots & \\
	&&& \frac{1}{j}
	\end{bmatrix} \in \RR^{(j-i+1) \times (j-i+1)}
	\numberthis
	\label{eq:D-tilde}
\end{align}
and $D_{s-1,s-1} \in \RR^{s-1 \times s-1}$ according to \eqref{eq:D-mat}. A (possibly repeated) Gram-Schmidt procedure follows to orthogonalize $y$ against $\{ q_1,\ldots,q_m\}$. The normalized version of this vector becomes $q_{m+1}$.

We note that since $A_0^{-1}$ does not change, we can compute the LU factorization before beginning the algorithm. In this way, the linear system can be solved efficiently in every iteration.
\begin{algorithm}
\SetAlgoLined
\SetKwInOut{Input}{input}\SetKwInOut{Output}{output}
\SetKwFor{FOR}{for}{do}{end}
\Input{$j$ the desired dimension of the Krylov subspace, \\ $A_i$, $i = 1,\ldots,j$ Taylor expansion coefficients of $A$, $b \in \CC^n$, \\ $\mu \in \CC \backslash \{ 0 \}$}
\Output{Approximate solution $x_j (\mu)$ of $A(\mu) x(\mu) = b$}
$\hat{c} = -A_0^{-1}b \in \CC^{n}$ \\  $Q_1 = \hat{c}/||\hat{c}||$ \\ $H_0 =$ \text{empty matrix}\\
\FOR{$m = 1,2, \dots, j$}{
    Let $\vectorize(X) = q_m \in \CC^{nm}$ \\
    Compute $\tilde{x}$ according to \eqref{eq:y-tilde} \\
    Compute $y = \vectorize{(\tilde{x},X D_{m+1,m})}$ according to \eqref{eq:D-mat}\\
    Expand $Q_m$ by $n$ rows of zeros \\
    Orthogonalize $y$ against $q_1, \dots, q_m$ by a Gram-Schmidt process: \label{alg:gs-line1} \\
    \nonl \Indp $h_m = Q_m^{T}y$ \\
    \nonl $y_{\perp} = y - Q_m h_m$ \\
    \Indm
        Possibly repeat Step \ref{alg:gs-line1} \\
    	Compute $\beta_m = ||y_{\perp}||$ \\
	Let $q_{m+1} = y_{\perp}/\beta_m$ \\
	Expand $Q_m$ into $Q_{m+1} = [Q_m, q_{m+1}]$ \\
	Let
	$\underline{H}_m =
    	\begin{bmatrix}
    	\underline{H}_{m-1} & h_m \\x
    	0 & \beta_m
    	\end{bmatrix} \in \CC^{(m+1) \times m}$
}
Return function handle (Theorem~\ref{solution})
\begin{align*}
    x_j (\mu): \mu \mapsto Q_j (1:n,:)  \argmin_{v \in \CC^j} \norm{ \left(\mu \underline{H}_j - \underline{I}_j \right) v - e_1 \norm{c} }
\end{align*}
where $\underline{I}_j$ is as in \eqref{eq:least-squares}
\caption{Infinite GMRES}
\label{alg:inf_gmres}
\end{algorithm}

\begin{algorithm}[ht]
\SetAlgoLined
\SetKwInOut{Input}{input}\SetKwInOut{Output}{output}
\SetKwFor{FOR}{for}{do}{end}
\Input{$j$ the desired dimension of the Krylov subspace, \\
$A_i$, $i=1,\ldots,s$, $U_i$, $i=s+1,\dots,j$ and $V^T$ as in \eqref{eq:A_reduced_rank}, \\ $b \in \CC^n$, $\mu \in \CC \backslash \{ 0 \}$}
\Output{Approximate solution $x_j (\mu)$ of $A(\mu) x(\mu) = b$}
\BlankLine
$\hat{c} = -A_0^{-1}b \in \CC^{n}$ \\  $Q_1 = \hat{c}/||\hat{c}||$ \\ $H_0 =$ empty matrix\\
\FOR{$m = 1,2, \dots, j$}{
    Compute $y$ using \eqref{eq:y-lr} $\textendash$ \eqref{eq:D-tilde} \\
    Expand $Q_m$ by: \\
    \nonl  \Indp $n$ rows of zeros if $m<s$\\
    \nonl  $p$ rows of zeros otherwise \\
    \Indm Orthogonalize $y$ against $q_1, \dots, q_m$ by a Gram-Schmidt process: \label{alg:gs-line} \\
    \nonl \Indp $h_m = Q_m^{T}y$ \\
    \nonl $y_{\perp} = y - Q_m h_m$ \\
        \Indm
        Possibly repeat Step \ref{alg:gs-line} \\
    	Compute $\beta_m = ||y_{\perp}||$ \\
	Let $q_{m+1} = y_{\perp}/\beta_m$ \\
	Expand $Q_m$ into $Q_{m+1} = [Q_m, q_{m+1}]$ \\
    Let
    $\underline{H}_m =
        \begin{bmatrix}
        \underline{H}_{m-1} & h_m \\
        0 & \beta_m
        \end{bmatrix} \in \CC^{(m+1) \times m}$}
Return function handle (Corollary \ref{low_rank_cor})
\begin{align*}
    x_j: \mu \mapsto Q_j (1:n,:)  \argmin_{v \in \CC^j} \norm{ \left(\mu \underline{H}_j - \underline{I}_j \right) v - e_1 \norm{c}}
\end{align*}
where $\underline{I}_j$ is in \eqref{eq:least-squares}
\caption{Low-Rank Infinite GMRES} \label{alg:low_rank_inf_gmres}
\end{algorithm}

The complete algorithm for low-rank infinite GMRES is included in Algorithm~\ref{alg:low_rank_inf_gmres}. An advantage of this method is that $Q_m$ grows only by $p$ rows after the first $s$ iterations, in contrast to Algorithm~\ref{alg:inf_gmres} where $Q_m$ grows by $n$ rows on each iteration.

\subsubsection{Tensor Infinite GMRES}\label{sect:tiar}
In \cite{TensorArnoldi}, a new method called tensor infinite Arnoldi (TIAR) was proposed. This method provided an equivalent factorization as in infinite Arnoldi method, but was improved in terms of memory and computation time.
Since this is essentially a memory efficient way to carry out
the infinite Arnoldi method, we can use this as  a drop-in replacement
in Algorithm~\ref{alg:inf_gmres}.
%We propose to use this way of generating the Arnoldi factorization within  Algorithm~\ref{alg:inf_gmres}.

More precisely, Algorithm~\ref{alg:inf_gmres} and infinite Arnoldi method generate the same basis matrix for the Krylov subspace. Directly applying \cite[Lemma~3.1]{TensorArnoldi}, we can generate a tensor representation of the basis matrix instead of explicitly forming it, then convert all necessary operations to work on this factorization. %This is the method TIAR.
We propose a new method which uses TIAR to form the Arnoldi factorization within Algorithm~\ref{alg:inf_gmres}, which we call Tensor infinite GMRES. This method is more efficient than Algorithm~\ref{alg:inf_gmres} in practice and leads to an equivalent solution.

\section{Convergence theory} \label{sec:convergence}
In the following we present a convergence characterization of
Algorithm~\ref{alg:inf_gmres}. The tensor variant in Section~\ref{sect:tiar}
is equivalent to Algorithm~\ref{alg:inf_gmres}  and convergence results for
the low-rank version Algorithm~\ref{alg:low_rank_inf_gmres}
can be derived analogously.

More precisely, we will now show that the convergence of Algorithm~\ref{alg:inf_gmres} can be described by the magnitude of the parameter $\mu$ and the smallest solutions to an associated NEP. Solutions to the NEP are given generally as $\lambda \in \CC$ s.t. $A(\lambda)v = 0$ for $A(\mu) \in \CC^{n \times n}$ and $v \in \CC^n \backslash \{0 \}$. Essentially, the reciprocal eigenvalues of the NEP play the same role eigenvalues play for the convergence of standard GMRES.

GMRES is a Krylov subspace method with a finite termination for which the residual vectors satisfy the following for $C \in \CC^{n \times n}$ nonsingular, $b \in \CC^n$,
\begin{align}
    \norm{r_k} = \min_ {x \in \mathcal{K}_k(C,b)} \norm{Cx-b} = \min_{p \in \mathcal{P}_k^0} \norm{p(C) b},
    \label{eq:residual}
\end{align}
where
\begin{align*}
	\mathcal{P}_k^0 = \{ \text{polynomials } p \text{ of degree } \leq k \text{ with } p(0)=1 \}.
\end{align*}

Various specializations of the min-max bound
\eqref{eq:residual}
can be found in the literature, see e.g., the summary \cite{Strakos:2012:KRYLOV}. We specialize a bound based on the $j+1$ largest eigenvalues of the matrix $A$,
where $j$ is a free parameter, typically
the number of outlier eigenvalues.
When $j=0$, this leads to the standard bound involving
the largest eigenvalue.

We need the following theorem, which is a generalization of Gelfand's Theorem\footnote{For any matrix norm $||\cdot||$, $\rho(A) = \lim_{k \rightarrow \infty} ||A^k||^{1/k}$.}.
\begin{theorem} \label{the:want_to_show}
Suppose $A \in \mathbb{C}^{n \times n}$. Assume $\gamma_1, \dots, \gamma_j$ are the $j$ largest eigenvalues and that they are semi-simple. Moreover, assume $|\gamma_i| < \min (|\gamma_1|, \dots, |\gamma_j|)$ for $i=j+1, \dots, n$.

Then,
\begin{align*}
    \lim_{k \xrightarrow{} \infty} ||(A - \gamma_1 I)(A - \gamma_2 I) \dots (A - \gamma_j I) A^k||^{\frac{1}{k}} = |\gamma_{j+1}|.
\end{align*}
\end{theorem}
Although Gelfand's theorem is well-known in many variations, we have not
found this particular variant in the literature and therefore provide a proof of Theorem~\ref{the:want_to_show} in Appendix~\ref{appendix:proof1}.

Let $j \in \NN^+$ be as in Theorem~\ref{the:want_to_show} for $\mathbf{B}_N$ as in \eqref{comp_matrix}. We define
%\footnote{We select a polynomial $q \in \mathcal{P}_k^0$ based on the outlier eigenvalues $\gamma_1,\ldots,\gamma_j$ of $\mathbf{B}_N$ which leads to a bound including $\gamma_{j+1}$.}
\begin{align*}
    q(z) := \frac{\left( \prod_{i=1}^j (z + 1 - \mu \gamma_i) \right)(z + 1)^{k-j}}{\prod_{i=1}^j (1 - \mu \gamma_i)},
    %\label{eq:zarantonello}
\end{align*}
which can be viewed as a generalization of the Zarantonello polynomial \cite[pg 201]{IteratMSparseSys}.

With this specific $q$, the bound \eqref{eq:residual} can be simplified if we use
 the matrix $\mu \mathbf{B}_N - I$ as in \eqref{eq:Blinsys}  and the right-hand side vector $c$ as in \eqref{eq:RHS}.
%
% Applying \eqref{eq:residual} with
%\eqref{eq:zarantonello} to the matrix defined in \eqref{eq:shifted-sys}
We obtain the bound
\begin{align*}
    \norm{r_k} \leq \norm{q(\mu \mathbf{B}_N - I)c} \leq \norm{\frac{\left( \prod_{i=1}^j (\mu \mathbf{B}_N - \mu \gamma_i) \right)(\mu \mathbf{B}_N)^{k-j} }{\prod_{i=1}^j (1- \mu \gamma_i)}} \cdot \norm{c},
\end{align*}
and therefore also
\begin{align*}
    \norm{r_k}^{\frac{1}{k}} \leq
    \frac{\norm{\mu^k}^{\frac{1}{k}}}{\norm{\prod_{i=1}^j (1 - \mu \gamma_i)}^{\frac{1}{k}}} \cdot
    \norm{\left( \prod_{i=1}^j (\mathbf{B}_N - \gamma_i) \right)(\mathbf{B}_N)^{k}}^{\frac{1}{k}}
      \cdot \norm{\mathbf{B}_N^{-j}}^{\frac{1}{k}} \cdot \norm{c}^{\frac{1}{k}}.
\end{align*}
Using Theorem~\ref{the:want_to_show}, we have
\begin{align*}
    &\norm{r_k}^{\frac{1}{k}} \leq |\mu| |\gamma_{j+1}| \text{ as } k \rightarrow \infty,
\end{align*}
and consequently
\begin{align}\label{eq:rk_bound}
    \norm{r_k} \leq (|\mu| |\gamma_{j+1}|)^k,
\end{align}
for sufficiently large $k$.

From equation \eqref{eq:rk_bound} we
conclude that the convergence factor bound
is proportional to both $\mu$ and $\gamma_{j+1}$. The value of $\gamma_{j+1}$ can be further interpreted as follows. The $\mathbf{B}_N$-matrix is a companion matrix with eigenvalues $\gamma_i$, and its reciprocal eigenvalues are solutions to a PEP. More specifically, $\gamma_i = 1/\lambda_i$ where $\lambda_i$ are s.t.
\begin{align} \label{eq:our-nep}
	-A(0)^{-1} A_N(\lambda_i) v_i = 0,
 \end{align}
 where $v_i \in \CC^n \backslash \{ 0 \}$ and $A_N$ represents a truncated Taylor series expansion of $A$. Under the assumption that the eigenvalues converge (as a function of the truncation) the values $\lambda_1,\ldots,\lambda_j$ will converge to the $j$ smallest eigenvalues of the nonlinear eigenvalue problem. Therefore, the convergence of Algorithm~\ref{alg:inf_gmres} for \eqref{eq:Blinsys} is determined by the largest reciprocal  eigenvalues of the nonlinear eigenvalue problem, in the same way that the largest eigenvalues of a matrix describe the convergence of GMRES.

\section{Simulations} \label{sec:simulations}
\subsection{Time-delay system}\label{sec:TDS}
We provide reproducible simulations from different applications
illustrating properties of the methods. All simulations were carried
with a sustem with a 2.3 GHz Dual-Core Intel Core i5 processor and 16 GB RAM
using Julia \cite{Bezanson2017}. The software for the simuluations are
available online\footnote{\url{https://github.com/siobhanie/InfGMRES}}.

As a first illustration we consider the dynamical system with delays described by
\begin{subequations}\label{eq:TDS}
\begin{eqnarray}
	\dot{x}(t) &=& A_0 x(t) + A_1 x (t - \tau) - b u(t) \\
	y(t) &=& C^T x(t),
\end{eqnarray}
\end{subequations}
where $A_0, A_1 \in \CC^{n \times n}$. For simplicity
we assume that the entire state is the output, i.e., $C = I \in \CC^{n \times n}$.
The vector $b \in \CC^n$ is the external force,
$x(t) \in \CC^n$ is the state vector,
$u(t)$ is the input, $y(t)$ is the output and $\tau > 0$ is the delay. We assume without loss of generality $\tau = 1$.
In the context of systems and control this is usually referred
as a time-delay system, see standard references for time-delay
systems, e.g. \cite{Gu:2003:STABILITY,Michiels2011,Michiels:2007:STABILITYBOOK}.
%\cite{Saadvandi:2012:DPA}.

The frequency domain formulation of \eqref{eq:TDS} relates
the input and the output as follows
%To study this system in the frequency domain, we apply a Fourier transform to the state equation above,
%\begin{align}
%	i \omega \hat{x}(\omega) &= A_0 \hat{x}(\omega) + A_1 e^{-i \omega} \hat{x}(\omega) - b \hat{u}(w)\\
%	\hat{y}(\omega) &= \hat{x}(\omega),
%\end{align}
%then relate the input and the output in the frequency domain as follows:
\begin{align*}
	\hat{y}(\omega) = H(i \omega) \hat{u}(\omega)
\end{align*}
where
\begin{align}\label{eq:TDS2}
	H(s) = (-sI + A_0 + A_1 e^{-s})^{-1} b.
\end{align}
The matrix $H(s)$ is called the transfer function and can be obtained by applying the Laplace transform to the state equation under the condition $x(0)=\boldsymbol{0}$.
Note that \eqref{eq:TDS2} is a parameterized linear
system of the form \eqref{eq:our-problem}. We use our approach
to evaluate the transfer function for many $s$-values.
%We are interested in solving the parameter dependent system $A(\mu) x(\mu) = b$ where $A(\mu) = H(\mu)$.

%The eigenvalues of the constant matrix $\mathbf{B}_N$ coming from the linearization of $A(\mu)$ (see Section~\ref{sec:introduction}) are the reciprocal eigenvalues of the PEP given in \eqref{eq:our-nep}.
The eigenvalues $\lambda_i$ of the PEP given in \eqref{eq:our-nep} are the reciprocals of the eigenvalues $\gamma_i$ of the constant matrix $\mathbf{B}_N$ coming from the linearization of $A(\mu)$ (see Section~\ref{sec:introduction}).
%Figure~\ref{fig:eigenvalues} shows the solutions to this PEP and Figure~\ref{fig:eigenvalues2} shows the eigenvalues of $\mathbf{B}_N$, plotted in the complex plane along with
Figure~\ref{fig:alleigvals}, we plot the solutions to the delay eigenvalue problem
\begin{equation}\label{eq:DEP}
 0=(-\lambda I+A_0+A_1e^{-\lambda})v
\end{equation}
closest to the origin (which were computing using the algorithm \cite{Jarlebring:2010:DELAYARNOLDI}).
Note that as $N\rightarrow\infty$ the eigenvalues of $\mathbf{B}_N$
approach the reciprocal eigenvalues of the delay eigenvalue problem.
For reference we plot also a circle centered around origin with radius $| \gamma_{j+1} |$, where $j=4$ and $\gamma_{j+1}$ is the non-outlier eigenvalue of largest magnitude, as described in Theorem~\ref{the:want_to_show}. A result of Cauchy's residue theorem and the principal of argument guarantees that within such a compact set, we have only a finite number of solutions to the nonlinear eigenvalue problem, unless $\mu = 0$ is a solution.

Figure~\ref{fig:con_figure1} and Figure~\ref{fig:con_figure2} show plots of iterations vs the norm of the relative residual when evaluating \eqref{eq:TDS2} with Algorithm~\ref{alg:inf_gmres}. We can see the convergence is proportional to $|\gamma_{j+1}|$ and $|\mu|$ with $j=4$, which illustrates the bound \eqref{eq:rk_bound}.

In Figure~\ref{fig:mu-vs-its} we show a plot of $\mu$ vs iterations to achieve a residual below $10^{-12}$ when evaluating \eqref{eq:TDS2}. We see that iterations required for convergence increases with $\mu$. Figure~\ref{fig:exp-conv} shows the observed and predicted convergence factors when evaluating \eqref{eq:TDS2} for different choices of $\mu$. The predicted convergence factor was calculated using the bound \eqref{eq:rk_bound} with $j=0$ outliers. The observed convergence factor measured by how much the residual decreased with each iteration of Algorithm~\ref{alg:inf_gmres} before convergence. In particular, we visualize the decrease at iteration $k$ which was least impressive, to represent a worst case scenario. More precisely, we plot $\mu$ vs $\rho$ where
\begin{align}\label{eq:rho}
\rho = \sup_{k} \frac{\norm{r_{k+1}}}{\norm{r_k}}.
\end{align}

\begin{figure}[ht]
    %\centering
    \begin{subfigure}[t]{0.45\columnwidth}
    %\centering
    \includegraphics{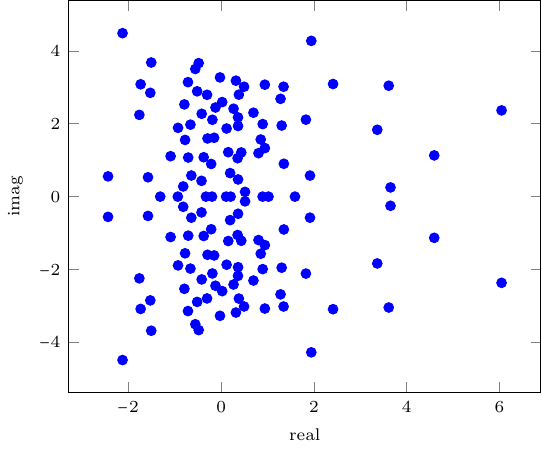}
    \caption{Solutions $\lambda$ to the delay eigenvalue problem \eqref{eq:DEP}}
    \label{fig:eigenvalues}
    \end{subfigure}
    \hfill
   \begin{subfigure}[t]{0.45\columnwidth}
    %\centering
     \includegraphics{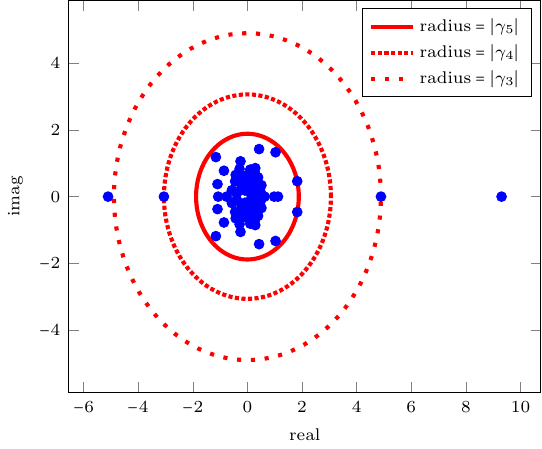}
    \caption{Reciprocal eigenvalues of the DEP \eqref{eq:DEP}, corresponding to the spectrum of $\mathbf{B}_N$ as $N\rightarrow\infty$}
    \label{fig:eigenvalues2}
    \end{subfigure}
    \caption{Solutions $\lambda_i$ to the delay eigenvalue problem
    %$-A(0)^{-1}(-\lambda I + A_0 + A_1 e^{-\lambda})v = 0$, $v \in \mathbb{C}^n \neq 0$
    and $\gamma_i = 1/\lambda_i$ for
    %of $\mathbf{B}_N$,
    $n=100$\label{fig:alleigvals}}
\end{figure}

\begin{figure}[ht]
     %\centering
     \begin{subfigure}[b]{0.45\columnwidth}
     %\centering
      \includegraphics{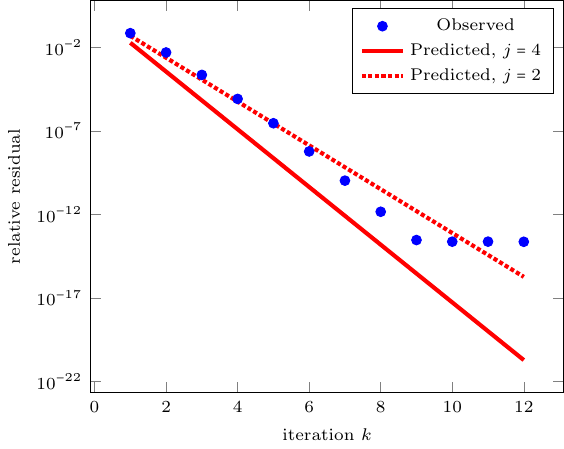}
     \caption{$\mu = .01$}
     \label{fig:con_figure1}
     \end{subfigure}
    \hfill
    \begin{subfigure}[b]{0.45\columnwidth}
    %\centering
     \includegraphics{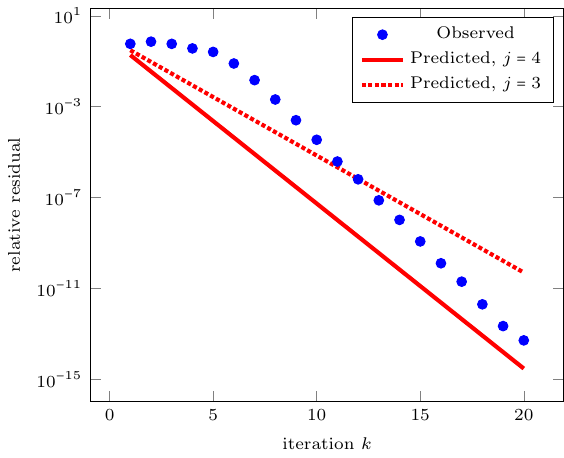}
    \caption{$\mu = .1$}
     \label{fig:con_figure2}
     \end{subfigure}
     \caption{Iterations vs norm of relative residual for evaluating \eqref{eq:TDS2} with Algorithm~\ref{alg:inf_gmres} plotted with bound predicted by \eqref{eq:rk_bound} with different choices outlier elimination $j$.
     % $n=100$.
     Note that, by construction, the prediction is a bound only for sufficiently large $k$.}
\end{figure}

\begin{figure}[ht]
    %\centering
    \begin{subfigure}[t]{0.45\columnwidth}
    %\centering
     \includegraphics{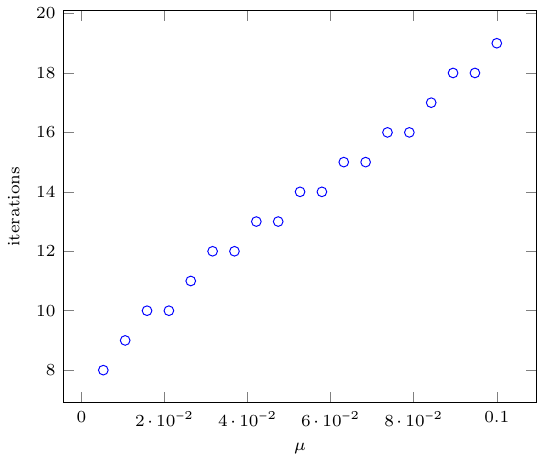}
    \caption{Iterations to achieve a relative residual below $10^{-12}$}
    \label{fig:mu-vs-its}
    \end{subfigure}
    \hfill
    \begin{subfigure}[t]{0.45\columnwidth}
    %\centering
     \includegraphics{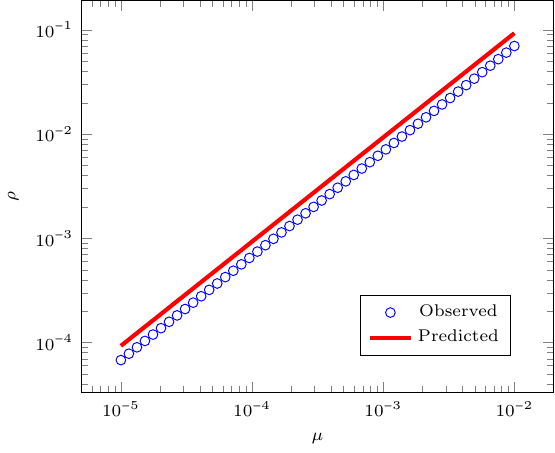}
    \caption{$\mu$ vs observed convergence factor $\rho$ as in \eqref{eq:rho} and predicted convergence factor as in \eqref{eq:rk_bound}}
    \label{fig:exp-conv}
    \end{subfigure}
    \caption{Convergence of Algorithm~\ref{alg:inf_gmres} for evaluating \eqref{eq:TDS2} for different values of $\mu$, $n=100$}
\end{figure}

\subsection{Absorbing boundary conditions}
The low-rank structure described in Section~\ref{sect:LR_infgmres}, arises
naturally from artificial boundary conditions similar to
how the structure arise for  NEPs \cite{doi:10.1002/nla.2043}.
Although the following technique is applicable for a wide class
of problems, for simplicity we illustrate for this specific boundary value problem:
\begin{subequations} \nonumber
\begin{eqnarray} \nonumber
   \bigg( \frac{\partial^2}{\partial x^2} + \big( 1 + \mu k(x) \big) ^2 + \beta (x) \bigg) u(x)  &=& h(x) \\\nonumber
   u(a) &=& 0 \\\nonumber
   u(c) &=& 0\nonumber
\end{eqnarray}
\end{subequations}
where
\begin{gather*}
    k(x) =
    \begin{cases}
    5+ \frac{10x}{b} \sin( \frac{\alpha \pi}{b} x) & x \in [0, \frac{b}{2}) \\
    5+10(1-\frac{x}{b}) \sin(\frac{\alpha \pi}{b} x)& x \in [\frac{b}{2},b) \\
    5 & x \geq b,
    \end{cases} \\
%\end{align*}
%\begin{align*}
	\beta (x) =
	\begin{cases}
	\sin(\frac{x-a}{b-a} 2 \pi) & x \in [a,b) \\
	0 & x \geq b,
	\end{cases} \\
%\end{align*}
%\begin{align*}
    h(x) =
    \begin{cases}
    \frac{(x-b)^2}{(a-b)^2} & x \in [a,b) \\
    0 & x  \geq b.
    \end{cases}
\end{gather*}
Plots of $k(x)$ and $h(x)$ follow in Figure~\ref{fig:potential} and Figure~\ref{fig:rhs} respectively. Note that $k,\beta,h$ are constant
in $[b,c]$ which allows us to do the following transformation.
We transform the problem on the interval $[b,c]$:
\begin{align*}
    \frac{d}{dx}
    \begin{bmatrix}
    u(x) \\
    u'(x)
    \end{bmatrix}
    =
    \begin{bmatrix}
    0 & 1 \\
    -(1 + \mu k(x))^2 - \beta(x) & 0
    \end{bmatrix}
    \begin{bmatrix}
    u(x) \\
    u'(x)
    \end{bmatrix}
    +
    \begin{bmatrix}
    0 \\
    h(x)
    \end{bmatrix}.
\end{align*}
Since $k(x) \equiv k_0$ where $k_0 = k(b)$, $\beta(x) \equiv 0$ and $h(x) \equiv 0$ on $[b,c]$, we can use the matrix exponential to solve the following differential equation on this interval, i.e.,
\begin{align*}
    \begin{bmatrix}
    u(x) \\
    u'(x)
    \end{bmatrix}
    =
    \text{exp} \left( (x-b)
    \begin{bmatrix}
    0 & 1 \\
    -(1 + k_0 \mu)^2 & 0
    \end{bmatrix} \right)
    \begin{bmatrix}
    u(b) \\
    u'(b)
    \end{bmatrix}.
\end{align*}
The boundary condition at $u(c)$ can be imposed as
\begin{align*}
    0 = u(c) =
    \begin{bmatrix}
    1 & 0
    \end{bmatrix}
    \begin{bmatrix}
    u(c) \\
    u'(c)
    \end{bmatrix}
    =
    \begin{bmatrix}
    1 & 0
    \end{bmatrix}
    \text{exp} \left( (c-b)
    \begin{bmatrix}
    0 & 1 \\
    -(1 + k_0 \mu)^2 & 0
    \end{bmatrix} \right)
    \begin{bmatrix}
    u(b) \\
    u'(b)
    \end{bmatrix}.
\end{align*}
We can compute the above matrix exponential using the formula for the matrix exponential of an antidiagonal two-by-two matrix. Thus, we obtain the relation
\begin{align}
    0 = g(\mu) u(b) + f(\mu) u'(b),
    \label{eq:boundary}
\end{align}
where
\begin{align*}
    g(\mu) &:= \cos \left( (c-b) (1 + k_0 \mu) \right) \\
    f(\mu) &:= \sin \left( (c-b) (1 + k_0 \mu) \right) \frac{1}{1 + k_0 \mu}.
\end{align*}

\begin{figure}[ht]
    %\centering
    \begin{subfigure}[t]{0.45\columnwidth}
    %\centering
     \includegraphics{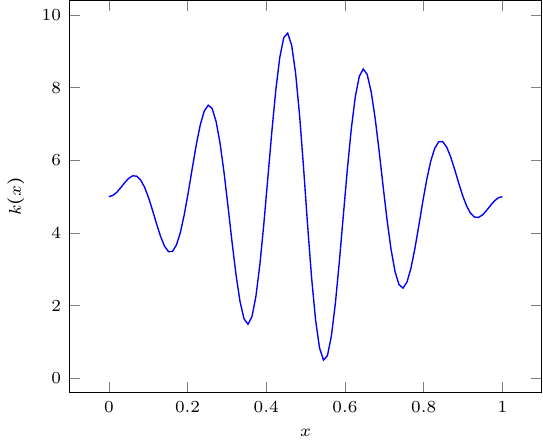}
    \caption{Potential function $k(x)$, $x \in [0,1]$}
    \label{fig:potential}
    \end{subfigure}
    \hfill
    \begin{subfigure}[t]{0.45\columnwidth}
    %\centering
     \includegraphics{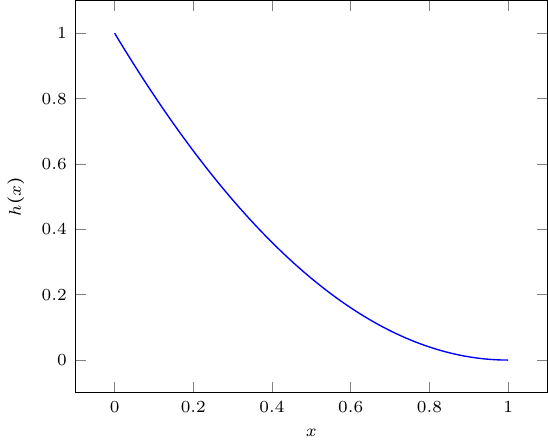}
    \caption{Right-hand side function $h(x)$, $x \in [0,1]$}
    \label{fig:rhs}
    \end{subfigure}
    \caption{Functions for the boundary value problem \eqref{eq:helmholtz_eq}}
\end{figure}

We note that a solution to the original boundary value problem will also satisfy \eqref{eq:boundary}. With this in mind, we split the domain into two parts and solve for $u(x)$ on $[a,b]$, i.e.,
\begin{subequations} \label{eq:helmholtz_eq}
\begin{eqnarray}
   \left( \frac{\partial^2}{\partial x^2} + \big( 1 + \mu k(x) \big)^2 + \beta(x) \right) u(x) &=& h(x) \\
   u(a) &=& 0 \\
   g(\mu) u(b) + f(\mu) u'(b) &=& 0,
\end{eqnarray}
\end{subequations}
a boundary value problem on the reduced domain with a Robin boundary condition at $x=b$. This is essentially a Dirichlet-to-Neumann map absorbing boundary condition, where the additional parameter appears in the operator, in this case as a scalar coefficient in the boundary condition.
The technique above can be seen as a special case of methods
in the field of artifical boundary condition.
See \cite{Appeloe2006,Berenger1994,HagstromRad} for literature on artificial boundary conditions. In the examples and plots which follow, we have used $[a,b) = [0,1)$, $[b,c] = [1,1.5]$ and $\alpha = 10$.

We discretize the problem as follows. Let $x_k = k \Delta x$, $k=1,\ldots,n$, and $\Delta x=1/n$ with $x_1 = \Delta x$ and $x_n = b$. To approximate the Robin boundary condition at $x=b$, we use a one-sided second order difference scheme, i.e.,
\begin{align*}
    0 = g(\mu) u(b) + f(\mu) \frac{1}{\Delta x} \big( \frac{3}{2}u(b) - 2u(x_{n-1}) + \frac{1}{2}u(x_{n-2}) \big) + \mathcal{O}(\Delta x^2).
\end{align*}
Thus, the discretized boundary value problem can be expressed as
\begin{align*}
	A_n(\mu) u_n(\mu) = h_n,
	%\label{eq:helmholtz-eq-to-solve}
\end{align*}
where
\begin{align*}
    A_n(\mu) = D_n + K_n(\mu) + L_n + X_n F(\mu) Y_n^T,
\end{align*}
with
\begin{align*}
    D_n &=
    \frac{1}{\Delta x^2}
    \begin{bmatrix}
    -2 & 1 & & & \\
    1 & \ddots & \ddots & & \\
    & \ddots & \ddots & \ddots & \\
    & & 1 & -2 & 1 \\
    0 & \cdots & 0 & 0 & 0
    \end{bmatrix} \in \mathbb{R}^{n \times n},
\end{align*}
\begin{gather*}
    K_n(\mu) = \text{diag} \big( \big(1 + \mu k(x_1) \big)^2, \ldots, \big(1 + \mu k(x_{n-1}) \big)^2, 0 \big) \in \mathbb{R}^{n \times n}, \\
    L_n = \text{diag} \big(\beta(x_1),\ldots,\beta(x_{n-1}),0 \big) \in \mathbb{R}^{n \times n}
\end{gather*}
and
\begin{gather*}
	X_n = \left[ e_n, e_n \right]  \in \mathbb{R}^{n \times 2}, \\Y_n  = \left[ e_n, \big( \frac{3}{2 \Delta x}e_n - \frac{2}{\Delta x} e_{n-1} + \frac{1}{2 \Delta x} e_{n-2} \big) \right]  \in \mathbb{R}^{n \times 2}, \\
         F(\mu) =
	\begin{bmatrix}
	g(\mu) & \\
	& f(\mu)
	\end{bmatrix} \in \mathbb{R}^{2 \times 2}, \\
	h_n = \left[ h(x_1), \ldots, h(x_n) \right]^T \in \mathbb{R}^{n \times 1}.
\end{gather*}
We note that this corresponds to a discretization of $A(\mu)$ with
a low-rank structure as given in \eqref{eq:A_reduced_rank}, where $U_i = X_n F^{(i)}(0)$, $V=Y_n$ and $s=2$. Therefore, all three
proposed algorithms are applicable to this problem.

\begin{figure}[ht]
     \begin{subfigure}[t]{0.45\columnwidth}
        %\pgfplotsset{scaled y ticks=false}
          \includegraphics{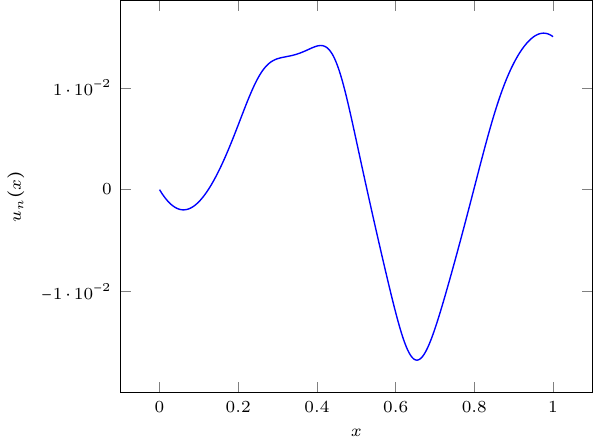}
         \caption{$\mu=1.6$}
         \label{fig:num-sol1}
       \end{subfigure}
       \hfill
     \begin{subfigure}[t]{0.45\columnwidth}
         \pgfplotsset{scaled y ticks=false}
          \includegraphics{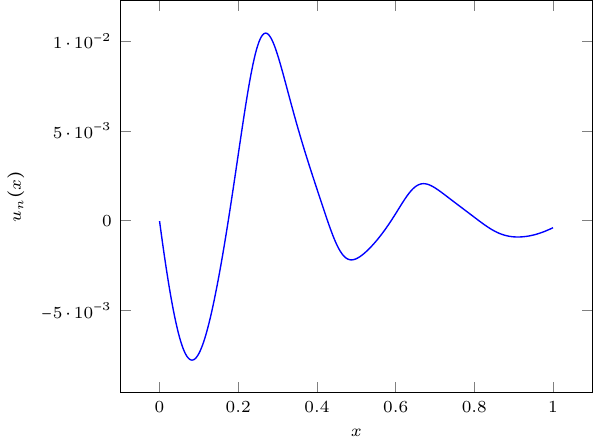}
         \caption{$\mu=2.5$}
         \label{fig:num-sol2}
     \end{subfigure}
     \caption{Numerical solution for \eqref{eq:helmholtz_eq}}
     \label{fig:num-sol-helm}
\end{figure}

\begin{figure}[ht]
     \begin{subfigure}[t]{0.45\columnwidth}
         \includegraphics{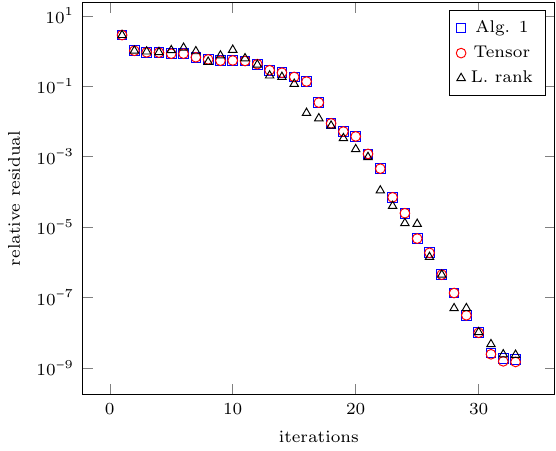}
        \caption{Convergence for \eqref{eq:helmholtz_eq}}
        \label{fig:conv1}
     \end{subfigure}
     \hfill
     \begin{subfigure}[t]{0.45\columnwidth}
         \includegraphics{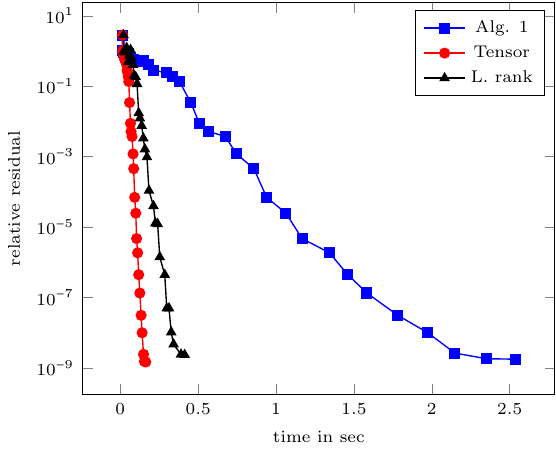}
        \caption{Simulation time for \eqref{eq:helmholtz_eq}}
         \label{fig:sim-time1}
     \end{subfigure}
    \caption{$\mu=1.6$, $n=5000$, condition number $\kappa(A_n(\mu)) = 5.621 \times 10^9$}
    \label{fig:sol-res-small}
\end{figure}

\begin{figure}[ht]
     \begin{subfigure}[b]{0.45\columnwidth}
         \includegraphics{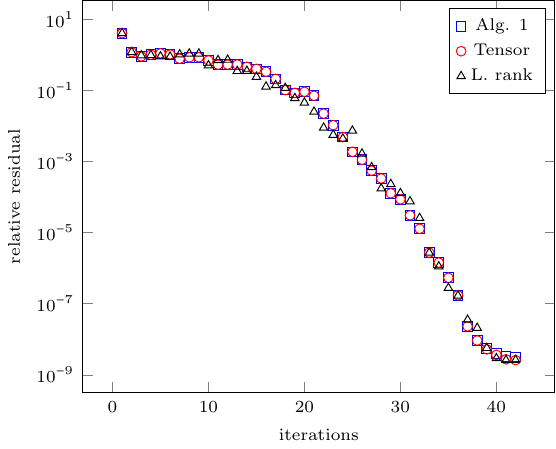}
        \caption{Convergence for \eqref{eq:helmholtz_eq}}
        \label{fig:conv2}
     \end{subfigure}
      \hfill
      \begin{subfigure}[b]{0.45\columnwidth}
         %\centering
         \includegraphics{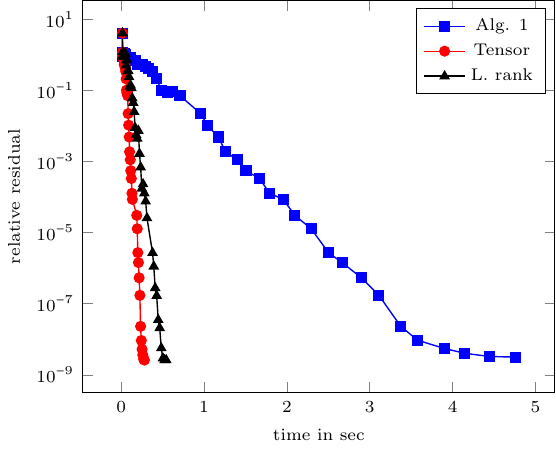}
        \caption{Simulation time for \eqref{eq:helmholtz_eq}}
        \label{fig:sim-time2}
     \end{subfigure}
    \caption{$\mu=2.5$, $n=5000$, condition number $\kappa(A_n(\mu)) = 2.398 \times 10^9$}
    \label{fig:sol-res-sim}
\end{figure}

In Figures~\ref{fig:num-sol1} and \ref{fig:num-sol2}, we see the numerical solution of the boundary value problem \eqref{eq:helmholtz_eq} on the domain $[a,b]$ for two different values of $\mu$, calculated with Algorithm~\ref{alg:low_rank_inf_gmres}. Figures~\ref{fig:conv1} and \ref{fig:conv2} show the convergence of Algorithm~\ref{alg:inf_gmres}, the tensor version of Algorithm~\ref{alg:inf_gmres} and Algorithm~\ref{alg:low_rank_inf_gmres} for solving \eqref{eq:helmholtz_eq}.
%In
%general, the error is smaller for smaller $\mu$, and we can therefore expect that Figure~\ref{fig:conv2} shows an approximate bound for all $\mu$ smaller than $0.2$.
Figures~\ref{fig:sim-time1} and \ref{fig:sim-time2} specify the error as a function of CPU-time for a given $\mu$, although after one run of the algorithm we have access to the solution approximation for many different $\mu$. We see that the tensor version of Algorithm~\ref{alg:inf_gmres} and Algorithm~\ref{alg:low_rank_inf_gmres} offer an improvement in CPU-time over Algorithm~\ref{alg:inf_gmres}, especially for larger $\mu$.

\section{A finite element discretization of Helmholtz equation}
In order to illustrate the competitiveness of our approach we consider a Helmholtz equation with a parameter dependent material coefficient and using a discretization with the finite element software FEniCS \cite{Alnaes:2015:FENICS}. Specifically, we consider the following Helmholtz equation with a homogeneous Dirichlet boundary condition
\begin{subequations} \label{eq:helmholtz_2d}
\begin{alignat}{2}
   \left( \nabla^2 + f_1(\mu) \big( 1 + \mu k(x) \big) ^2 + f_2(\mu) \beta (x) \right) u(x) &= h(x) \qquad &&x \in \Omega \\
   u(x) &= 0 &&x \in \partial \Omega,
\end{alignat}
\end{subequations}
where $x=(x_1,x_2)$, $\Omega$ is as described on pp. 37-39 in \cite{Precondbook} and
\begin{gather*}
    k(x) =
    \begin{cases}
    1 + (x_1) \sin(\alpha \pi x_1) & x_1 \in [0, \frac{1}{2}) \\
    1 + (1 - x_1) \sin(\alpha \pi x_1) & x_1 \in [\frac{1}{2},1]
    \end{cases}, \\
    h(x) = e^{-\alpha x_1^2}, \text{ } \beta (x) = \sin(2 \pi x_1), \\
    f_1(\mu) = \mu, \text{ } f_2(\mu) = \sin(\mu).
\end{gather*}

\begin{figure}[ht]
     \begin{subfigure}[b]{0.45\columnwidth}
		\includegraphics{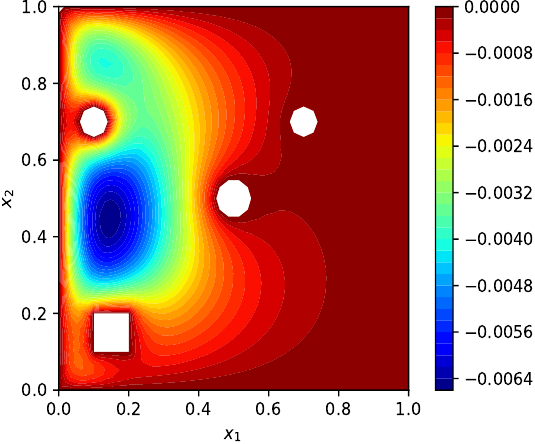}
     		\caption{$\mu=.1$}
     		\label{fig:2d-helm1}
     \end{subfigure}
     \hfill
     \begin{subfigure}[b]{0.45\columnwidth}
                \includegraphics{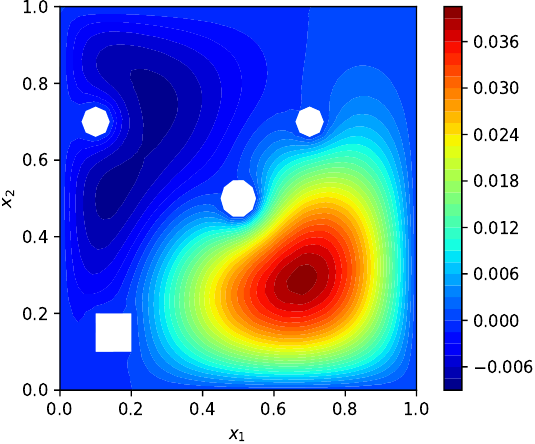}
	       \caption{$\mu=3$}
	       \label{fig:2d-helm2}
     \end{subfigure}
     \caption{Numerical solutions for \eqref{eq:helmholtz_2d}, $n=2436$}
     \label{fig:2d-plots}
\end{figure}

\begin{figure}
      %\centering
      \includegraphics{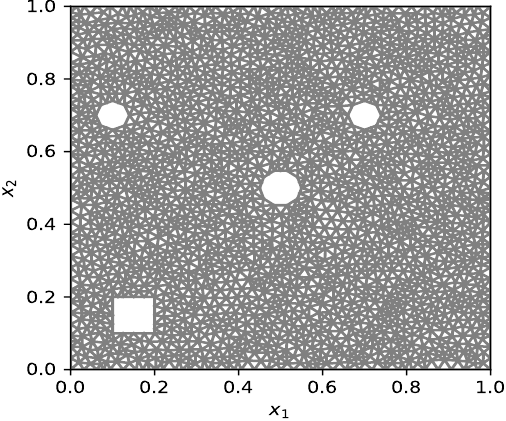}
      \caption{Mesh for FEM discretization, $n=2436$}
     \label{fig:2d-mesh}
\end{figure}

\begin{figure}[ht]
     %\centering
      \begin{subfigure}[b]{0.45\columnwidth}
      	%\centering
	\includegraphics{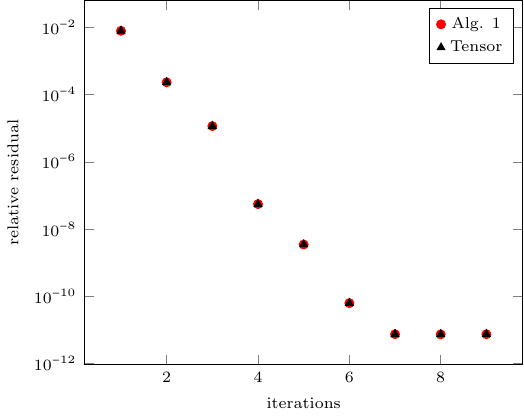}
        \caption{$\mu=.1$}
        \label{fig:2d-conv-1a}
     \end{subfigure}
     \hfill
	\begin{subfigure}[b]{0.45\columnwidth}
	 %\centering
        \includegraphics{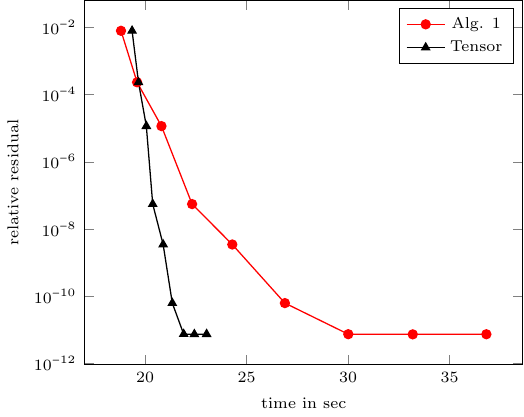}
        \caption{$\mu=.1$}
        \label{fig:2d-conv-1b}
     \end{subfigure}

      \begin{subfigure}[b]{0.45\columnwidth}
         %\centering
        \includegraphics{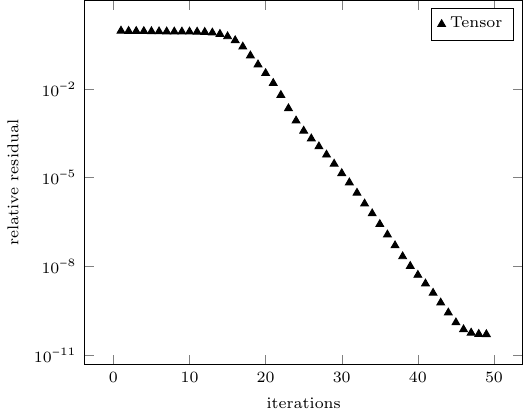}
        \caption{$\mu=3$}
        \label{fig:2d-conv-2a}
     \end{subfigure}
     \hfill
	\begin{subfigure}[b]{0.45\columnwidth}
	 %\centering
        \includegraphics{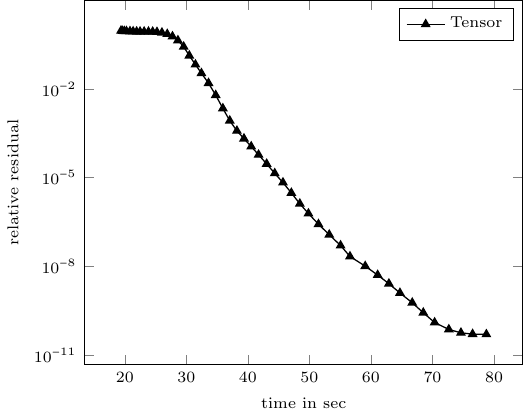}
        \caption{$\mu=3$}
        \label{fig:2d-conv-2b}
     \end{subfigure}
     \caption{Algorithm~\ref{alg:inf_gmres} and tensor variant applied to \eqref{eq:helmholtz_2d}, $n=741294$}
     \label{fig:2d-conv}
\end{figure}

Figures~\ref{fig:2d-helm1} and \ref{fig:2d-helm2} show the solutions to \eqref{eq:helmholtz_2d} on $\Omega$. We display the solution for two different values of $\mu$. We include also a plot of the mesh used to generate the finite element matrices needed for the linearization of this problem in Figure~\ref{fig:2d-mesh}. Figures~\ref{fig:2d-conv-1a}-\ref{fig:2d-conv-1b} show the performance of Algorithm~\ref{alg:inf_gmres} and its tensor version for solving the same problem for a small value of $\mu$. We see the benefit of the tensor variation in terms of time taken to build the Krylov basis matrix. Figures~\ref{fig:2d-conv-2a}-\ref{fig:2d-conv-2b} show the performance of the same problem with a larger value of $\mu$. We note that here we omit simulations with Algorithm~\ref{alg:inf_gmres} due to insufficient memory.

\section{Conclusion and outlook} \label{sec:conclusions}
The result of this paper is a new Krylov-subspace method to solve parameter-dependent systems of the form $A(\mu) x(\mu) = b$ for many values of $\mu$ simultaneously, where $A$ depends nonlinearly on $\mu$. We have constructed a companion linearization where $\mu$ appears only linearly and constructed a basis for the Krylov subspace in an efficient way without introducing truncation error. Numerical experiments verify the convergence of our methods is predicted by the magnitude of the parameter $\mu$ and the solutions to the corresponding NEP.

We have shown how to specialize this method to solve a specific discretized boundary value problem where the higher order terms in the Taylor series are of a certain form due to a Robin boundary condition on one end. In this way we have incorporated the structure of the problem into the design of the algorithm.

There are several variants of the infinite Arnoldi method, e.g. the Chebyshev version \cite{Jarlebring:2012:INFARNOLDI} and restarting variations \cite{Jarlebring:2014:SCHUR}. These strategies could be applied to the methods presented in this paper, but this would require adaptation based on the specific structure of the problem and further analysis.

\section*{Acknowledgements}
We are grateful for Prof. Tobias Damm, TU Kaiserslautern for providing crucial ideas for the proof of this version of Gelfand's lemma. We thank Prof. Kirk Soodhalter, Trinity College, for discussions of Krylov methods
for shifted linear systems.
\appendix

\section{Proof of Theorem~\ref{the:want_to_show}} \label{appendix:proof1}
Since $\gamma_1, \dots, \gamma_j$ are semi-simple, a Jordan decomposition can be expressed as
\begin{align*}
    A = V \text{diag}(\gamma_1, \gamma_2, \dots, \gamma_j, J) V^{-1} = V \Gamma V^{-1}.
\end{align*}
We have
\begin{align}
\centering
    \norm{(A - \gamma_1 I)(A - \gamma_2 I) \cdots (A - \gamma_j I) A^k} &\leq \kappa(V) \norm{(\Gamma - \gamma_1 I)(\Gamma - \gamma_2 I) \cdots (\Gamma - \gamma_j I) \Gamma^k} \label{first} \\
    \norm{(A - \gamma_1 I)(A - \gamma_2 I) \cdots (A - \gamma_j I) A^k} &\geq \frac{1}{\kappa(V)} \norm{(\Gamma - \gamma_1 I)(\Gamma - \gamma_2 I) \cdots (\Gamma - \gamma_j I) \Gamma^k}  \label{second}
\end{align}
where \eqref{second} follows from properties of singular values, i.e., $||B|| \geq \sigma_{\text{min}} (B)$.

Due to the placement of the zeros on the diagonals of the matrices $(\Gamma - \gamma_i I)$, $i=1,\dots,j$ and the upper triangular structure of all the matrices, we have the relation
\begin{align*}
    \left( \prod_{i=1}^j (\Gamma - \gamma_i I)\right) \Gamma^k &=
    \prod_{i=1}^j \left(
    \begin{bmatrix}
    \gamma_1 - \gamma_i & & & & \\
    & \gamma_2 - \gamma_i & & &  \\
    & & \ddots & & \\
    & & & \gamma_j -  \gamma_i & \\
    & & & & J - \gamma_i I \\
    \end{bmatrix} \right)
    \begin{bmatrix}
    \gamma_1^k & & & & \\
    & \gamma_2^k & & & \\
    & & \ddots & & \\
    & & & \gamma_j^k & \\
    & & & & J^k \\
    \end{bmatrix} \\
    &=
    \begin{bmatrix}
    0 & \\
    & \prod_{i=1}^j (J - \gamma_i I)
    \end{bmatrix}
    \begin{bmatrix}
    0 & \\
    & J^k
    \end{bmatrix}.
\end{align*}
Thus, we have
\begin{align*}
    \norm{ \left( \prod_{i=1}^j (\Gamma - \gamma_i I) \right) \Gamma^k }&=
    \norm{ \begin{bmatrix}
    0 & \\
    & \prod_{i=1}^j (J - \gamma_i I)
    \end{bmatrix}
    \begin{bmatrix}
    0 & \\
    & J^k
    \end{bmatrix} } \leq \alpha \norm{J^k}_2 \\
    \norm{ \left( \prod_{i=1}^j (\Gamma - \gamma_i I) \right) \Gamma^k } &=
    \norm{ \begin{bmatrix}
    0 & \\
    & \prod_{i=1}^j (J - \gamma_i I)
    \end{bmatrix} \begin{bmatrix}
    0 & \\
    & J^k
    \end{bmatrix} } \geq \beta \norm{J^k}_2,
\end{align*}
where
\begin{align*}
    \alpha =  \sigma_{\text{max}} \left( \prod_{i=1}^j (\Gamma - \gamma_i I) \right) > 0 \\
    \beta =  \sigma_{\text{min}} \left( \prod_{i=1}^j (\Gamma - \gamma_i I) \right) > 0.
\end{align*}
We note that we have $\alpha > 0$ since the matrix is non-zero. We have $\beta > 0$ since $\gamma_1, \dots, \gamma_j$ are not eigenvalues of $J$ by assumption, and therefore all singular values are positive (non-zero).

So, we have
\begin{align*}
    \norm{(A - \gamma_1 I)(A - \gamma_2 I) \cdots (A - \gamma_j I) A^k}^{\frac{1}{k}} &\leq (\kappa_2(V) \alpha)^{\frac{1}{k}} \norm{J^k}_2^{\frac{1}{k}} \xrightarrow{k \xrightarrow{} \infty} |\gamma_{j+1}| \\
    \norm{(A - \gamma_1 I)(A - \gamma_2 I) \cdots (A - \gamma_j I) A^k}^{\frac{1}{k}} &\geq \bigg(\frac{\beta}{\kappa_2(V)}\bigg)^{\frac{1}{k}} \norm{J^k}_2^{\frac{1}{k}} \xrightarrow{k \xrightarrow{} \infty} |\gamma_{j+1}|,
\end{align*}
and thus
\begin{align*}
    \lim_{k \xrightarrow{} \infty} \norm{(A - \gamma_1 I)(A - \gamma_2 I) \cdots (A - \gamma_j I) A^k}^{\frac{1}{k}} = |\gamma_{j+1}|.
\end{align*}

\bibliographystyle{plain}
\bibliography{siobhanbib,eliasbib}

\end{document}